\documentclass[12pt]{amsart}
\usepackage{amssymb}
\usepackage{hyperref}

\theoremstyle{plain}
\newtheorem{theorem}{Theorem}[section]

\newtheorem{lemma}[theorem]{Lemma}

\theoremstyle{definition}

\newtheorem{remark}[theorem]{Remark}
\newtheorem{question}[theorem]{Question}

\theoremstyle{remark}

\numberwithin{equation}{section}

\newcommand{\R}{\mathbb R}
\newcommand{\C}{\mathbb C}

\newcommand{\GL}{\operatorname{GL}}
\newcommand{\SL}{\operatorname{SL}}
\newcommand{\Ot}{\operatorname{O}}
\newcommand{\SO}{\operatorname{SO}}
\newcommand{\SU}{\operatorname{SU}}

\newcommand{\Sp}{\operatorname{Sp}}

\newcommand{\spp}{\mathfrak{sp}}

\newcommand{\op}{\operatorname}

\newcommand{\spec}{\operatorname{Spec}}

\newcommand{\tr}{\operatorname{tr}}

\title[Spectral uniqueness]{Spectral uniqueness of bi-invariant metrics on symplectic groups}
\author{Emilio~A.~Lauret}
\address{Institut f\"ur Mathematik, Humboldt-Universit\"at zu Berlin, 10099 Berlin, Germany.}
\address{Permanent affiliation: CIEM--FaMAF (CONICET), Universidad Nacional de C\'ordoba, Medina Allende s/n, Ciudad Universitaria, 5000 C\'ordoba, Argentina.}
\email{elauret@famaf.unc.edu.ar}
\thanks{This research was supported by the Alexander von Humboldt Foundation, and grants from CONICET and FONCyT}
\date{June 9, 2018}

\begin{document}

\begin{abstract}
In this short note, we prove that a bi-invariant Riemannian metric on $\mathrm{Sp}(n)$ is uniquely determined by the spectrum of its Laplace-Beltrami operator within the class of left-invariant metrics on $\mathrm{Sp}(n)$. 
In other words, on any of these compact simple Lie groups, every  left-invariant metric which is not right-invariant cannot be isospectral to a bi-invariant metric. 
The proof is elementary and uses a very strong spectral obstruction proved by Gordon, Schueth and Sutton.
\end{abstract}

\maketitle

\section{Introduction}\label{sec:intro}

Given $(M,g)$ a compact Riemannian manifold, we call the \emph{spectrum} of $(M,g)$, denoted by $\spec(M,g)$, the spectrum of the Laplace--Beltrami operator $\Delta_g$. 
Two compact Riemannian manifolds are called \emph{isospectral} if their spectra coincide. 

It is well known that $\spec(M,g)$ does not determine the isometry class of $(M,g)$, because of a large number of non-isometric isospectral examples.
However, despite this, it is to be expected that Riemannian manifolds with very special geometric properties are spectrally distinguishable from other Riemannian manifolds.
For example, Tanno~\cite{Tanno73} showed that any round sphere of dimension $\leq 6$ cannot be isospectral to any non-isometric orientable Riemmanian manifold. 
For a recent account on spectrally distinguishable Riemannian manifolds, we refer the reader to \cite[\S1]{GordonSchuethSutton}.

Since symmetric spaces have very special geometric properties, one may expect to decide whether $(M,g)$ is a Riemannian symmetric space from $\spec(M,g)$. 
This apparently simple problem has no solutions except for Tanno's result on round spheres of dimension up to $6$. 
Consequently, it seems reasonable to restrict the space of metrics. 
In this article we wish to consider the following question.
\begin{question}\label{question}
Is a bi-invariant metric on a compact connected semisimple Lie group $G$ spectrally distinguishable within the space $\mathcal M^G$ of left-invariant metrics on $G$?
\end{question}
Irreducible Riemannian symmetric spaces are divided in two types. 
One of them is given by bi-invariant metrics on compact simple Lie groups. 
Furthermore, any metric in $\mathcal M^G$ is homogeneous. 
Although homogeneous Riemannian manifolds have nice geometric properties, there exist isospectral deformations of homogeneous metrics (\cite{Schueth01a}, \cite{Proctor05}) and pairs of isospectral homogeneous manifolds with interesting properties (\cite{Sutton02}, \cite{AnYuYu13}). 
We also note that, without the semisimple assumption on $G$, there are examples of isospectral flat metrics on any torus of dimension $\geq4$ (see for instance \cite[page \texttt{xxix}]{ConwaySloane-book}).

There are several advances of local type concerning Question~\ref{question}.
In \cite{Schueth01a}, among many other nice results, Schueth proved that a bi-invariant metric is infinitesimally spectrally rigid, that is, it cannot be continuously isospectrally deformed. 
When $G$ is simple, Gordon and Sutton~\cite{GordonSutton10} proved that any metric in the space of naturally reductive left-invariant metrics $\mathcal M_{\text{nat}}^G$ is spectrally isolated in $\mathcal M_{\text{nat}}^G$.
Furthermore, they also proved that any family of mutually isospectral compact symmetric spaces is finite.

In the author's opinion, the best results toward providing an answer to Question~\ref{question} were obtained in \cite{GordonSchuethSutton}. 
In this article, Gordon, Schueth and Sutton proved that a bi-invariant metric is spectrally isolated in $\mathcal M^G$. 
More precisely, there is a neighborhood in $\mathcal M^G$ around the fixed bi-invariant metric containing no metric isospectral to the fixed one besides itself. 
The topology considered in $\mathcal M^G$ is induced by the natural topology in the space of inner products on the Lie algebra of $G$. 
Actually, their result is much stronger since it shows that a finite part of the spectrum ignoring multiplicities suffices to prove the spectral isolation. 
When $G$ is simple, the first two non-zero eigenvalues are sufficient.   
Moreover, \cite[Prop.~3.1]{GordonSchuethSutton} provides a very strong restriction on the spectrum of the Laplace--Beltrami operator of any left-invariant metric on a simple Lie group (see Theorem~\ref{thm:GSS} below).

To the author's knowledge, Question~\ref{question} has been answered only in the cases $\SU(2)\simeq S^3$ by Tanno~\cite{Tanno73} and $\SO(3)$ by Schmidt and Sutton~\cite{SchmidtSutton13}\footnote{This article can be found on Schmidt's web page and dates from October 16th 2014. By a personal communication with the second named author of \cite{SchmidtSutton13}, in the near future, they will post a revised version with an extended result on the arXiv.}.
Actually, Schmidt and Sutton proved that any left-invariant metric on any of these two groups is uniquely determined by its spectrum in $\mathcal M^G$ (see \cite{Lauret-SpecSU(2)} for a recent alternative proof).
It was shown that the first four heat invariants cannot coincide for non-isometric pairs. 
They made use of the fact that, for $G=\SU(2)$ or $\SO(3)$, the space $\mathcal M^G/\sim$ (left-invariant metrics on $G$ up to isometry) has dimension $3$. 
This low dimension is not usual since $\dim (\mathcal M^G/\sim )\geq m^2-\binom{m}{2}-m$ for any $m$-dimensional compact simple Lie group $G$. 
For example, the next simplest case is $G=\SU(3)$, where $\dim (\mathcal M^G/\sim)=28$.

We now formulate our main result. 

\begin{theorem}\label{thm:main}
If a left-invariant metric on $\Sp(n)$ is isospectral to a bi-invariant metric on $\Sp(n)$, then they are isometric. 
\end{theorem}

The proof utilizes in an essential way the mentioned spectral obstruction \cite[Prop.~3.1]{GordonSchuethSutton} (see also Theorem~\ref{thm:GSS}) due to Gordon, Schueth and Sutton.
The main argument studies the multiplicity of the first non-zero eigenvalue (see Theorem~\ref{thm:stronger}).

\section{Spectra of left invariant metrics}\label{sec:preliminaries}

Let $G$ be a compact connected semisimple Lie group of dimension $m$. 
It is well known that left-invariant metrics on $G$ are in correspondence with inner products on the Lie algebra $\mathfrak g$ of $G$. 
We fix $\langle \cdot,\cdot\rangle_I$ any $\op{Ad}(G)$-invariant inner product on $\mathfrak g$, for instance, minus the Killing form.
Let $\{X_1,\dots,X_m\}$ be an orthonormal basis of $\mathfrak g$ with respect to $\langle \cdot,\cdot \rangle_I$. 
For $A=(a_{i,j})\in\GL(m,\R)$, we denote by $\langle \cdot,\cdot\rangle_A$ the inner product on $\mathfrak g$ satisfying that $\{Y_1,\dots,Y_m\}$ is an orthonormal basis, where $Y_j=\sum_{i=1}^{m}a_{i,j}X_i$ for any $j$. 
One can check that $\langle \cdot,\cdot \rangle_{A} = \langle \cdot,\cdot \rangle_{B}$ if and only if $A=BP$ for some $P\in\Ot(m)$. 
Consequently, the space of left-invariant metrics on $G$ is in correspondence with $\GL(m,\R)/\Ot(m)$. 
For $A\in\GL(m,\R)$, we denote by $(G,g_A)$ the Riemannian manifold $G$ endowed with the left-invariant metric $g_A$ corresponding to $\langle\cdot,\cdot\rangle_A$.

Let $(R,L^2(G))$ be the right-regular representation of $G$, that is, for each $g\in G$, $R_g:L^2(G)\to L^2(G)$ is unitary given by $f\mapsto (R_g\cdot f)(x)=f(xg)$ for $x\in G$ and $f\in L^2(G)$. 
The Peter-Weyl Theorem ensures the equivalence 
\begin{equation}\label{eq:PeterWeyl}
L^2(G)\simeq \bigoplus_{\pi\in \widehat G} V_\pi\otimes V_\pi^*
\end{equation}
as $G$-modules.
Here, $\widehat G$ denotes the unitary dual of $G$, the action of $G$ on $V_\pi\otimes V_\pi^*$ is given by $g\cdot (v\otimes\varphi )= (\pi(g)v)\otimes\varphi$, and the embedding $V_\pi\otimes V_\pi^*\hookrightarrow C^\infty(G)$ is given by $v\otimes \varphi\mapsto f_{v\otimes\varphi}(x) = \varphi(\pi(x)v)$ for  $x\in G$.

Let $\Delta_A$ be the Laplace--Beltrami operator of $(G,g_A)$. 
One has that (c.f.\ \cite[Lem.~1]{Urakawa79}) 
\begin{equation}\label{eq:Laplacian}
\Delta_A\cdot f_{v\otimes\varphi} = f_{(\pi(-C_A)v)\otimes\varphi},
\end{equation}
where $C_A=\sum_{j=1}^m Y_j^2\in U(\mathfrak g_\C)$ with $\{Y_1,\dots,Y_m\}$ an orthonormal basis of $\mathfrak g$ with respect to $\langle\cdot,\cdot\rangle_A$.
Consequently, if $\lambda_{1}^{\pi,A},\dots, \lambda_{d_\pi}^{\pi,A}$ denote the eigenvalues of the finite-dimensional linear operator $\pi(-C_A):V_\pi\to V_\pi $ with $d_\pi=\dim V_\pi$, then 
\begin{equation}\label{eq:spec_A}
\spec (\Delta_A) = \bigcup_{\pi\in\widehat G} \big\{\!\big\{
\underbrace{\lambda_1^{\pi,A},\dots, \lambda_1^{\pi,A}}_{d_\pi\text{-times}},\dots, \underbrace{\lambda_{d_\pi}^{\pi,A},\dots,\lambda_{d_\pi}^{\pi,A}}_{d_\pi\text{-times}} 
\big\}\!\big\}.
\end{equation}

When $A=I$, $C_I$ is the Casimir element up to a positive constant, which lies in the center of the universal enveloping algebra of $\mathfrak g_\C$, so $\pi(-C_I)$ commutes with the action of $G$.
Thus $\pi(-C_I)=\lambda^\pi \operatorname{Id}_{V_\pi}$ by Schur's Lemma.  
Hence, 
\begin{equation}\label{eq:spec_I}
\spec (\Delta_I) = \bigcup_{\pi\in\widehat G} \big\{\!\big\{
\underbrace{\lambda^{\pi},\dots, \lambda^{\pi}}_{d_\pi^2\text{-times}} 
\big\}\!\big\}.
\end{equation}

We fix a root system $\Sigma(\mathfrak g_\C,\mathfrak h)$, where $\mathfrak h$ is a Cartan subalgebra of $\mathfrak g_\C$. 
Let $\rho$ denote half the sum of positive roots. 
Assume that the inner product $\langle\cdot,\cdot\rangle_I$ on $\mathfrak g$ is a negative multiple of the Killing form. 
Let us denote again by $\langle\cdot,\cdot\rangle_I$ the induced inner products on $\mathfrak g_\C$ and on $\mathfrak h^*$. 
For any $\pi\in \widehat G$, we have that (see for instance \cite[Prop.~5.28]{Knapp-book-beyond})
\begin{equation}\label{eq:eigenvalue}
-\lambda^{\pi} = \langle \mu_\pi+\rho,\mu_\pi+\rho \rangle_I - \langle \rho,\rho \rangle_I  = \langle \mu_\pi+2\rho,\mu_\pi \rangle_I , 
\end{equation}
where $\mu_\pi$ denotes the highest weight of $\pi$. 
It is important to note that $\langle\cdot,\cdot\rangle_A$ is negative definite on $\op{span}_\R\{\alpha:\alpha\in \Sigma(\mathfrak g_\C,\mathfrak h)\}$. 

For $A\in \GL(m,\R)$, set $\|A\|=\tr(A^t A)^{1/2} =(\sum_{i,j}a_{i,j}^2)^{1/2}$.

\begin{theorem}[Gordon, Schueth, Sutton~\cite{GordonSchuethSutton}]  \label{thm:GSS}
Let $G$ be a compact simple Lie group, let $g_I$ be a bi-invariant metric on $G$ and let $g_A$ be a left-invariant metric defined as above with $A\in \GL(m,\R)$ and $\det(A)\geq1$ (i.e.\ $\op{vol}(G,g_A)\leq \op{vol}(G,g_I)$). 
Then
\begin{equation}\label{eq:GSS}
\tr (\Delta_A|_{W}) = \frac{\|A\|^2}{m} \tr(\Delta_I|_{W}),
\end{equation}
where $\frac{\|A\|^2}{m}\geq1$ with equality if and only if $A\in\Ot(m)$, 
for every finite dimensional subspace $W$ of $L^2(G)$ which is invariant under the right-regular representation of $G$ and on which $G$ acts non-trivially.
\end{theorem}

We next apply this theorem to $W=V_{\pi}\otimes V_\pi^*$ for each $\pi\in\widehat G$ non-trivial. 
In the notation of \eqref{eq:spec_A} and \eqref{eq:spec_I},  $\tr (\Delta_A|_{V_{\pi}\otimes V_\pi^*} ) =d_\pi \tr(\pi(-C_A)) =d_\pi  \sum_{j=1}^{d_\pi} \lambda_{j}^{\pi,A}$ and $\tr(\Delta_I|_{V_\pi\otimes V_\pi^*})=d_\pi\tr(\pi(-C_I)) =d_\pi^2 \;\lambda^{\pi}$, thus 
\begin{equation}\label{eq:GSS2}
\tr (\pi(-C_A))  =  \frac{\|A\|^2}{m} \tr(\pi(-C_I)). 
\end{equation}
This identity will be the main tool in the proof of Theorem~\ref{thm:main}.

\section{Proof of the main theorem}\label{sec:proof}
We set $G=\Sp(n)$, thus $G$ is a simply connected compact simple Lie group for every $n\geq1$ of dimension $m:=n(2n+1)$.
We fix the Cartan subalgebra $\mathfrak h$ of $\mathfrak g_\C=\spp(n,\C)$ and the associated root system $\Sigma(\mathfrak g_\C,\mathfrak h)$ as in \cite[\S II.1]{Knapp-book-beyond}, which is a standard way. 
In particular, $\{\varepsilon_1,\dots,\varepsilon_n\}$ is a basis of $\mathfrak h^*$ and the positive roots are $\Sigma^+:=\{\varepsilon_i\pm\varepsilon_j:1\leq i<j\leq n\}\cup\{2\varepsilon_j:1\leq j\leq n\}$.
Furthermore, the corresponding fundamental weights are $\{\omega_1,\dots,\omega_n\}$ where $\omega_p=\varepsilon_1+\dots+\varepsilon_p$.
Since any dominant weight is a non-negative integer combination of fundamental weights, then any dominant weight has the form $\sum_{j=1}^n a_j\varepsilon_j$ for some integers $a_1,\dots,a_n$ satisfying $a_1\geq a_2\geq \dots\geq a_n\geq0$. 
Since $G$ is simple, any $\op{Ad}(G)$-invariant inner product on $\mathfrak g$ is a negative multiple of the Killing form. 
We fix $\langle \cdot,\cdot\rangle_I$ as the negative multiple of the Killing form such that its bilinear extension to $\mathfrak h^*$, again denoted by $\langle \cdot,\cdot\rangle_I$, satisfies $\langle \varepsilon_i,\varepsilon_j\rangle_I=-\delta_{i,j}$.

By the Highest Weight Theorem, the irreducible representations of $G$ are in correspondence with the dominant weights.
Given $\mu$ a dominant weight, we denote by $\pi_\mu$ the irreducible representation of $G$ with highest weight $\mu$.

\begin{lemma}\label{lem:dimensiones}
We have that $\lambda^{\pi_{\varepsilon_1}}<\lambda^\pi$ for every irreducible representation $\pi$ of $\Sp(n)$ with $\pi\not\simeq 1_g,\pi_{\varepsilon_1}$.
Furthermore, for $n=2$, $n=3$, $n=4$ or $n\geq5$, any irreducible representation $\pi$ of $\Sp(n)$ satisfying $\dim V_\pi \leq (\dim V_{\pi_{\varepsilon_1}})^2$ is in Table~\ref{table:n=2}, \ref{table:n=3}, \ref{table:n=4} or Table~\ref{table:n>=5} respectively.
\end{lemma}

\begin{table}
\begin{minipage}[t]{0.4\textwidth}
	\caption{Case $n=2$.}\label{table:n=2}
	\begin{center}
$
\begin{array}{ccc}
\mu_\pi &\lambda^\pi & \dim V_\pi \\ \hline 
		0&0&1\\
\varepsilon_1 & 5&4 \\
\varepsilon_1+\varepsilon_2 & 8&5\\
2\varepsilon_1 & 12&10 \\
2\varepsilon_1+\varepsilon_2 & 15&16\\
2\varepsilon_1+2\varepsilon_2 & 20&14
\end{array}
$
\end{center}
\end{minipage}
\quad
\begin{minipage}[t]{0.4\textwidth}
	\caption{Case $n=3$.}\label{table:n=3}
	\begin{center}
		$
		\begin{array}{ccc}
		\mu_\pi &\lambda^\pi & \dim V_\pi \\ \hline 
		0&0&1\\
		\varepsilon_1 & 7&6 \\
		\varepsilon_1+\varepsilon_2 & 12&14\\
		\varepsilon_1+\varepsilon_2+\varepsilon_3 & 15&14\\
		2\varepsilon_1 & 16&21
		\end{array}
		$
	\end{center}
\end{minipage}
\\[5mm]
\begin{minipage}[t]{0.4\textwidth}
\caption{Case $n=4$.}\label{table:n=4}
\begin{center}
$
\begin{array}{ccc}
\mu_\pi &\lambda^\pi & \dim V_\pi \\ \hline 
0&0&1\\
\varepsilon_1 & 9&8 \\
\varepsilon_1+\varepsilon_2 & 16&27\\
\varepsilon_1+\varepsilon_2 +\varepsilon_3 & 21&48\\
\varepsilon_1+\varepsilon_2 +\varepsilon_3+\varepsilon_4 & 24&42\\
2\varepsilon_1 & 20&36 
\end{array}
$
\end{center}
\end{minipage}
\quad
\begin{minipage}[t]{0.4\textwidth}
\caption{Case $n\geq5$.}\label{table:n>=5}
\begin{center}
$
\begin{array}{ccc}
\mu_\pi & \lambda^\pi & \dim V_\pi \\ \hline 
0&0&1\\
\varepsilon_1 &2n+1& 2n \\
\varepsilon_1+\varepsilon_2 &4n& (n-1)(2n+1)\\
2\varepsilon_1 &4(n+1)& n(2n+1)
\end{array} 
$
\end{center}
\end{minipage}
\end{table}

\begin{proof}
From \eqref{eq:eigenvalue}, for any dominant weight $\mu=\sum_{j=1}^n a_j\varepsilon_j$, we see that
\begin{equation}\label{eq:lambda^mu}
\lambda^{\pi_\mu} = \sum_{j=1}^m a_j(a_j+2(n+1-j)). 
\end{equation} 
It follows immediately that $\lambda^{\pi_{\varepsilon_1}} = 2n+1<\lambda^{\pi_\mu}$ for every $\mu\neq0,\varepsilon_1$, which is the first assertion. 
Furthermore, one can easily check that $\lambda^{\pi_{2\varepsilon_1}} = 4(n+1)$,  $\lambda^{\pi_{\omega_p}} = p(2n+2-p)$ for $1\leq p\leq n$, and the rest of the highest weights appearing in Table~\ref{table:n=2}.

The Weyl Dimension Formula $\dim V_{\pi_\mu } = \prod_{\alpha\in \Sigma^+} \frac{\langle\mu+\rho,\alpha\rangle}{\langle\rho,\alpha\rangle}$ for $\mathfrak g_\C=\spp(n,\C)$ becomes 
\begin{align}\label{eq:Weyldimformula}
\dim V_{\pi_\mu } 
&= \left(\prod_{i=1}^{n}\frac{\langle\mu+\rho, 2\varepsilon_i\rangle}{\langle \rho, 2\varepsilon_i\rangle}\right) \left(\prod_{i=1}^{n-1} \prod_{j=i+1}^{n} \frac{\langle\mu+\rho, \varepsilon_i+\varepsilon_j\rangle}{\langle\rho, \varepsilon_i+\varepsilon_j\rangle}
\frac{\langle\mu+\rho, \varepsilon_i-\varepsilon_j\rangle}{\langle\rho, \varepsilon_i-\varepsilon_j\rangle}
\right),  
\end{align}
where $\rho=\sum_{j=1}^n (n+1-j) \varepsilon_j$. 
Straightforward calculations show that 
\begin{align}\label{eq:dimke_1omega_p}
\dim V_{\pi_{k\varepsilon_1}} &= \binom{k+2n-1}{k},
&
\dim V_{\pi_{\omega_p}} &= \frac{2n+2-2p}{2n+2-p} \, \binom{2n+1}{p},
\end{align}
for every $k\geq0$ and $1\leq p\leq n$. 
In particular, $\dim V_{\pi_{k\varepsilon_1}} \geq \binom{2n+2}{3}>4n^2$ for every $k\geq3$ and $n\geq2$, and $\dim V_{\pi_{\omega_p}}>4n^2$ for every $p\geq3$ and $n\geq5$.
Furthermore, it is a simple matter to check that the dimensions shown in the tables are correct by using \eqref{eq:Weyldimformula}. 

Each dominant weight is uniquely a linear combination $\mu = \sum_{j=1}^n b_j\omega_j$ of the fundamental weights, where $b_j$ is a non-negative integer for all $j$. 
The \emph{width} of $\mu$, denoted by $\op{wd}(\mu)$, is defined to be $\sum_{j=1}^n b_j$. 
Now, \cite[Lemmas 2.1 and 2.2]{GoldsteinGuralnickStong17} implies that 
\begin{equation}
\dim V_{\pi_\mu} \geq \dim V_{\pi_{\op{wd}(\mu)\varepsilon_1}}. 
\end{equation}
It follows that $\dim V_{\pi_\mu}>4n^2 $ for all $\mu$ with $\op{wd}(\mu)\geq3$. 
We already checked in the previous paragraph whether $\dim V_{\omega_j} > 4n^2$.
It only remains to consider dominant weights with width equal to $2$, which is left to the reader. 
Note that the case $n=2$ is done since all dominant weights with width $2$ are already present in Table~\ref{table:n=2}. 
\end{proof}


We are now in position to prove the main theorem for $n\geq2$, which follows immediately from the following stronger result. 
We recall that the case $n=1$ was already shown since $\Sp(1)\simeq \SU(2)\simeq S^3$.

\begin{theorem}\label{thm:stronger}
Let $g_0$ and $g$ be left-invariant metrics on $\Sp(n)$, $n\geq2$, such that $\op{vol}(G,g)=\op{vol}(G,g_0)$ and $g_0$ is also right-invariant.
If the least non-zero eigenvalue $\lambda_1(G,g_0)$ of $\Delta_{g_0}$ is in $\spec(G,g)$ with the same multiplicity as in $\spec(G,g_0)$, then $(G,g_0)$ and $(G,g)$ are isometric.  
\end{theorem}

\begin{proof}
By rescaling the metrics, since the volume is a spectral invariant, we can assume that $g_0=g_I$ and $g=g_A$ for some $A\in\SL(m,\R)$. 

We first show that it is sufficient to prove that 
\begin{equation}\label{eq:suff}
\tr (\pi(-C_A)) = \tr(\pi(-C_I))
\end{equation}
for some non-trivial representation $\pi\in \widehat G$. 
Indeed, if this is the case, then \eqref{eq:GSS2} gives $\|A\|^2=m$.
We claim that $A\in\SO(m)$, which implies that $\langle\cdot,\cdot\rangle_A = \langle\cdot, \cdot\rangle_I$, thus $(G,g_A)$ and $(G,g_I)$ are isometric as asserted. 
To show the claim, let $\sigma_1,\dots,\sigma_m$ denote the eigenvalues of the positive definite symmetric matrix $A^tA$.
The equality holds in the inequality of arithmetic and geometric means since $1= \det(A^t A)^{1/m} = ({\sigma_1\dots \sigma_m})^{1/m} \leq  (\sigma_1+\dots+\sigma_m)/m = \|A\|^2 /m =1$.
Consequently, $\sigma_1=\dots=\sigma_m=1$, $A^tA$ is similar to the identity matrix, so $A\in \SO(m)$. 
The rest of the proof consists in showing \eqref{eq:suff} for $\pi=\pi_{\varepsilon_1}$.

By \eqref{eq:spec_I} and Lemma~\ref{lem:dimensiones}, the smallest non-zero eigenvalue of $\Delta_I$ is $\lambda^{\pi_{\varepsilon_1}}=2n+1$ with multiplicity $4n^2$. 
Since $\spec(\Delta_A)=\spec(\Delta_I)$, there are $\pi_1,\dots,\pi_r\in\widehat G$ and integers $j_1,\dots,j_r$ satisfying  $1\leq j_h\leq d_{\pi_h}=\dim V_{\pi_h}$,  $\lambda_{j_h}^{\pi_h,A} = \lambda^{\pi_{\varepsilon_1}}$ for every $1\leq h\leq r$ and 
\begin{equation}\label{eq:integerequation}
4n^2 = \sum_{h=1}^r \dim V_{\pi_h}. 
\end{equation}
The second assertion in Lemma~\ref{lem:dimensiones} ensures that $\pi_1,\dots,\pi_h$ appear in Tables~\ref{table:n=2}--\ref{table:n>=5}.
Moreover, $\pi_h\not\simeq 1_G$ for all $h$ since $\lambda_1^{1_G,A}= 0$ because $\dim V_{1_G}=1$. 
The rest of the proof will be divided in the cases $n=2$, $n=3$, $n=4$ and $n\geq5$.
We recall that the case $n=1$ is included in Tanno's result mentioned in the introduction since $\Sp(1)\simeq \SU(2)$.

When $n\geq5$, Table~\ref{table:n>=5} implies that there are $a,b,c$ non-negative integer numbers such that $r=a+b+c$, $\pi_h\simeq \pi_{\varepsilon_1}$ for $a$ choices of $h$, $\pi_h\simeq \pi_{\varepsilon_1+\varepsilon_2}$ for $b$ choices of $h$, $\pi_h\simeq \pi_{2\varepsilon_1}$ for $c$ choices of $h$, and
\begin{align}
(2n+1)(2n-1)+1 =4n^2 & = a\dim V_{\pi_{\varepsilon_1}} + b\dim V_{\pi_{\varepsilon_1+\varepsilon_2}}+c\dim V_{\pi_{2\varepsilon_1}}\\
&= 2na+(n-1)(2n+1)b+n(2n+1)c. \notag
\end{align}
It follows easily that $2n+1$ divides $2na-1$, which implies that $a=2n$ and $b=c=0$. 
Consequently, $r=2n$, $\pi_h\simeq \pi_{\varepsilon_1}$ for all $h$, $\lambda_{j}^{\pi_{\varepsilon_1},A} = \lambda^{\pi_{\varepsilon_1}}$ for all $j$, so \eqref{eq:suff} holds for $\pi=\pi_{\varepsilon_1}$. 

When $n=3$, Table~\ref{table:n=3} gives that there are $a,b,c,d$ non-negative integer numbers satisfying  
\begin{equation}
36 =  d_{\pi_{\varepsilon_1}} a+  d_{\pi_{\varepsilon_1+\varepsilon_2}}b + d_{\pi_{\varepsilon_1+\varepsilon_2+\varepsilon_3}} c+  d_{\pi_{2\varepsilon_1}} d
=6a+14b+14c+21d,
\end{equation}
which immediately implies that $a=6$, $b=c=d=0$.
Similarly as above, we obtain \eqref{eq:suff}.

We now assume $n=4$. 
By Table~\ref{table:n=4}, there are $a,b,c,d,e$ non-negative integers such that 
\begin{equation}
64 = 8a+27b+48c+42d+36e.
\end{equation}
In this case, besides the trivial solution $a=8$, $b=c=d=e=0$ which immediately implies \eqref{eq:suff}, we only have the solution $a=2$, $c=1$, $b=d=e=0$. 
We assume this last possibility to obtain a contradiction.
We have that $r=3$, $\pi_1\simeq \pi_2\simeq \pi_{\varepsilon_1}$, $\pi_3\simeq \pi_{\varepsilon_1+\varepsilon_2+\varepsilon_3}$, $\lambda_{j}^{\pi_{\varepsilon_1},A} = \lambda^{\pi_{\varepsilon_1}}$ for exactly two choices of $j$, and $\lambda_{j}^{\pi_{\varepsilon_1+\varepsilon_2+\varepsilon_3},A} = \lambda^{\pi_{\varepsilon_1}}$ for exactly one choice of $j$.
We claim that the last condition is impossible. 
Indeed, the representation $\pi_{\varepsilon_1+\varepsilon_2+\varepsilon_3}$ is symplectic (see for instance \cite[Ch.~VI (5.3)]{BrockerDieck}), thus any eigenvalue of $\pi_{\varepsilon_1+\varepsilon_2+\varepsilon_3}(-C_A)$ has even multiplicity.
Consequently $\lambda_{j}^{\pi_{\varepsilon_1+\varepsilon_2+\varepsilon_3}, A} = \lambda^{\pi_{\varepsilon_1}}$ holds for an even number of choices of $j$.

We conclude the proof by considering the case $n=2$. 
Table~\ref{table:n=2} ensures that there are $a,b,c,d,e$ non-negative integers satisfying that 
\begin{equation}
16 = 4a+5b+10c+16d+14e.
\end{equation}
The trivial solution $a=4$, $b=c=d=e=0$ implies \eqref{eq:suff}. 
Let us check that the only other solution $d=1$, $a=b=c=e=0$ is impossible. 
The reason is the same as in the previous case.
Indeed, one can check that the representation $\pi_{2\varepsilon_1+\varepsilon_2}$ is symplectic, thus $\lambda_{j}^{\pi_{2\varepsilon_1+\varepsilon_2},A} = \lambda^{\pi_{\varepsilon_1}}$ holds necessary for an even number of choices of $j$. 
\end{proof}

\begin{remark}
One can check that, for any other compact simple Lie group $G$, equation \eqref{eq:integerequation} has at least two solutions. 
This explains why this proof does not work for any other case. 
\end{remark}

\section*{Acknowledgments}

The author would like to thank Dorothee Schueth for interesting discussions concerning this topic, 
to Jorge Lauret for extensive correspondence, and 
to the anonymous referees for giving him very helpful comments.
The author also wishes to thank the Alexander von Humboldt Foundation for financial support and the Humboldt Universit\"at zu Berlin for hospitality.

\bibliographystyle{plain}

\end{document}